\documentclass{article}
\usepackage{amssymb}
 \usepackage{amsthm}
\usepackage{amsmath,amssymb,amsopn,amsfonts,mathrsfs,amsbsy,amscd}
\usepackage{longtable}
\usepackage{multirow}
\usepackage[latin1]{inputenc}
\newcommand{\lr}{\longrightarrow}

\newcommand{\R}{\mathbb{R}}
\newcommand{\ita}{\textit}

\newcommand{\G}{{\mathfrak{g}}}

\newcommand{\h}{{\mathfrak{h}}}

\newcommand{\e}{\check{e}}
\newtheorem{Def}{Definition}
\newtheorem{theo}{Theorem}
\newtheorem{pr}{Proposition}
\newtheorem{Le}{Lemma}
\newtheorem{co}{Corollary}
\newtheorem{exem}{Example}
\newtheorem{remark}{Remark}

\title{ Symplectic Novikov  Lie Algebras}

\author{T. Ait Aissa and M. W. Mansouri\\Universit\'e Ibno Tofail\\ Facult\'e des Sciences. Laboratoire L.A.G.A\\ K\'enitra-Maroc\\e-mail: mansourimohammed.wadia@uit.ac.ma\\tarik.aitaissa @uit.ac.ma}

\begin{document}
\maketitle

\begin{abstract}
	It is well known that a symplectic Lie algebra admit a left symmetric product. In this work, we study the case where this product is Novikov, we show that the left-symmetric product associated to the symplectic Lie algebra is Novikov if and only if it is associative. In this case, the symplectic Lie algebra is called \ita{symplectic Novikov Lie algebra} (SNLA).  We show that any  SNLA is completely reducible two-step solvable. The classification of  four dimensional and  nilpotent   six dimensional, also some  methods for building large classes of examples, are presented. Finally,  we give a geometric study of the affine connection associated with an SNLA.
\end{abstract}

key words:
	Symplectic Lie algebras, Novikov algebra,  Symlectic connection.\\
AMS Subject Class (2010):  53D05,  17B60,  17B30.

\section{Introduction}
Novikov algebras arise in many areas of mathematics and physics. The Novikov structures appearing  in connection with the Poisson brackets of hydrodynamic type \cite{B-N}.
The study of Novikov algebras  was initiated by Zelmanov \cite{Z} then  developed by several authors \cite{B-D}, for  their classification problem see \cite{B-M}, \cite{B-M2} and \cite{B-G}.
The study of symplectic Lie groups (algebras) was developed by Bon-Yao Chu \cite{C}, Lichnerowitz, Medina and Ph. Revoy in \cite{L-M} and \cite{M-R}.

A finite-dimensional algebra $(\G,.)$ over $\R$ is called \ita{left-symmetric} if it satisfies the identity
\begin{equation}\label{1}
ass(x,y,z)=ass(y,x,z)\qquad \forall x,y,z \in\G,
\end{equation}
where $ass(x,y,z)$  denotes the associator $ass(x,y,z)=(x.y ).z-x.(y.z)$. In this case, the commutator $[x,y]= x.y-y.x$
defines a bracket that makes $\G$  a Lie algebra.
Clearly, each \ita{associative algebra} product (i.e. $ass(x,y,z)=0$, $\forall x,y,z\in\G$) is  a left symmetric product.

The left-symmetric algebra is called \ita{Novikov}, if 
\begin{equation}\label{2}
(x.y).z = (x.z).y\qquad \forall x,y,z \in\G,
\end{equation}
is satisfied. Let $\mathrm{L}_x$ and $\mathrm{R}_x$ denote the left and right multiplications by the element
$x\in\G$, respectively. The identity $(1)$ is now equivalent to the formula
\begin{equation*}
[\mathrm{L}_x,\mathrm{L}_y]=\mathrm{L}_{[x,y]}\qquad \forall x,y\in G
\end{equation*}
or in other words, the linear map $\mathrm{L} : \G \longrightarrow{End}(\G)$ is a representation of Lie algebras.

The identity $(2)$ is equivalent to each of the following identities
\begin{equation}\label{N}
[\mathrm{R}_x,\mathrm{R}_y]=0\qquad \forall x,y\in G
\end{equation}
\begin{equation}
\mathrm{L}_{x.y}=\mathrm{R}_y\circ \mathrm{L}_x\qquad \forall x,y\in G.
\end{equation}

For more details on left-symmetric algebras, we refer the reader to the survey article \cite{B} and the references therein (See as well \cite{H}).

A \ita{symplectic Lie algebra} $(\G,\omega)$ is a real Lie algebra with a skew-symmetric non-degenerate bilinear form 
$\omega$ such that for any $x,y,z\in\G$,
\[\omega([x,y],z)+\omega([y,z],x)+\omega([z,x],y)=0\]
this is to say, $\omega$ is a non-degenerate $2$-cocycle for the scalar cohomology of $\G$. We call $\omega$ a symplectic form of $\G$.  Note that in such case, $\G$ must be of even dimension. A fundamental example of symplectic Lie algebras are the Frobenius Lie algebras,  i.e. Lie algebras admitting a non-degenerate exact 2-form. Symplectic Lie algebras are in one-to-one correspondence with simply connected Lie groups with left invariant symplectic forms. The geometry of symplectic Lie groups is an active field of research.
There are several results on  the construction of symplectic Lie groups (see \cite{M-R}, \cite{D-M},  \cite{B-C}, \cite{F}) and some classifications in low dimension (\cite{O}, \cite{K-G-M}, \cite{F}).

It is known that  (see \cite{C} and \cite{M-R}) the product given by
\begin{equation}\label{5}
\omega(x.y,z)=-\omega(y,[x,z])\qquad\forall x,y\in\G
\end{equation}

induces a left symmetric algebra structure that satisfies $x.y-y.x=[x,y]$ on $\G$, we say that the left symmetric product is \ita{associated with the symplectic Lie algebra} $(\G,\omega)$.
Geometrically, this is equivalent to existence in a symplectic Lie group a \ita{left-invariant affine structure} (a left-invariant linear connection   with zero torsion and zero curvature). It is easy to  see \cite{M-R} that  the left symmetric product  associated with symplectic Lie algebra $(\G,\omega)$ satisfies
\begin{equation}\label{6}
\omega(x.y,z)=\omega(x,z.y)\qquad\forall x,y,z\in\G
\end{equation}
and 
\begin{equation}\label{7}
\omega(x.y,z)+\omega(y.z,x)+\omega(z.x,y)=0 \quad\forall x,y,z\in\G.
\end{equation}
The aim of this    work is to study the case where the left symmetric product associated with symplectic Lie algebra $(\G,\omega)$ is of Novikov.

\textit{Notations}:  For $\{e_i\}_{1\leq i\leq n}$  a basis of $\G$, we denote by $\{e^i\}_{1\leq i\leq n}$ the
dual basis on $\G^\ast$ and  $e^{ij}$  the two-form $e^i\wedge e^j\in\wedge^2\G^*$. For  an endomorphism $f : \G\lr \G$  we denote by $f^t : \G^*\lr \G^*$ its dual and by $f^*: \G\lr \G$ the symplectic adjoint endomorphism defined by
\begin{equation*}
\omega(f(x),y)=\omega(x,f^*(y)) \qquad\forall x,y\in\G.
\end{equation*}

The paper is organized as follows: In section $2$, we show that the left-symmetric product associated to $(\G,\omega)$ is Novikov if and only if it is associative and we give some algebraic consequences.  In section $3$, we give  a classification  of $4$-dimensional and  nilpotent  $6$-dimensional SNLA, then  we characterize the left-symmetric algebras $\h$ which defines ana SNLA SNLA structures on symplectic cotangent Lie algebra $(\h^\ast\times_{\ell}\h,\omega)$, this construction provides many examples of SNLA. In section $4$, we study the effect of symplectic reduction as well as the effect of symplectic oxidation on NSLAs and  we show that every SNLA is completely reducible. The last section studies the affine structure $\nabla$ associated with an SNLA we show that it is bi-invariant and completed if and only if the Lie algebra $\G$ is nilpotent. We also give a geometric study of the symplectic connection associated with $\nabla$.

\section{Definition and first properties}

\begin{Def}
	A \ita{symplectic Novikov  Lie algebra} (SNLA), is a symplectic  Lie algebra $(\G,\omega)$ with 	the associated  left symmetric product is Novikov.
\end{Def}
Our first result is the following
\begin{theo}\label{t1}
	Let $(\G,\omega)$ be a symplectic Lie algebra.  Then the left symmetric product associated with $(\G,\omega)$ is Novikov if  and only if it  is associative.
\end{theo}
\begin{proof} Suppose that the left symmetric product $"."$ associated with $(\G,\omega)$ is Novikov.
	By using $(\ref{5})$ and $(\ref{6})$, we have for every $x$, $y$, $z$, and $t\in\G$	
	\begin{eqnarray*}
		\omega((x.y).z,t)=\omega((x.z).y,t)&\Leftrightarrow&-\omega(z,[x.y,t])=\omega(x.z,t.y)\\
		&\Leftrightarrow&\omega(z,[x.y,t])=\omega(z,[x,t.y])\\
		&\Leftrightarrow&\omega(z,(x.y).t-t.(x.y)-x.(t.y)+(t.y).x)=0\\
		&\Leftrightarrow&\omega(z,(x.t).y-x.(t.y)+(t.x).y-t.(x.y))=0\\
		&\Leftrightarrow&\omega(z,2ass(t,x,y))=0.
	\end{eqnarray*}
	Then the product $"."$ is associative. Conversely, suppose that the left symmetric product associated with $(\G,\omega)$ is associative.  By using $(\ref{5})$ and $(\ref{6})$, we have for every $x$, $y$, $z$, and $t\in\G$
	\begin{eqnarray*}
		\omega((x.y).z,t)=\omega(x.(y.z),t)&\Leftrightarrow&\omega(y,[x,t.z])=\omega(y,[x,t].z)\\
		&\Leftrightarrow&\omega(y,x.(t.z)-(t.z).x-(x.t).z+(t.x).z)=0\\
		&\Leftrightarrow&\omega(y,(t.x).z-(t.z).x)=0.	
	\end{eqnarray*}
	Then the product $"."$ is Novikov.
\end{proof}

\begin{remark}
	The product $"."$ is Novikov and associative, this gives the formulas
	\begin{equation}\label{R8}
	\mathrm{R}_{[x,y]}=0, \qquad\forall x, y\in\G.
	\end{equation}
	\begin{equation}\label{R9}
	\mathrm{ad}_{[x,y]}=\mathrm{L}_{[x,y]}=[\mathrm{L}_x,\mathrm{L}_y],\qquad\forall x, y\in\G.
	\end{equation}
\end{remark}	

Note  that a Lie algebra admitting a Novikov structure must be solvable, see \cite{B-D}. In the case of SNLA we have the result.

\begin{co}
	Let $(\G,\omega)$ be an SNLA. Then, the associated
	Lie algebra  is two-step solvable.
\end{co}
\begin{proof}
	For any $x$, $y$, $z$, $t\in\G$, by use of that the product associated  with $(\G,\omega)$ is Novikov and associative, we have
	\begin{eqnarray*}
		[[x,y],[z,t]]	&=&[x.y-y.x,z.t-t.z]\\
		&=&(x.y-y.x).(z.t-t.z)-(z.t-t.z).(x.y-y.x)\\
		&=&x.y.z.t-x.y.t.z-y.x.z.t+y.x.t.z-z.t.x.y+z.t.y.x+t.z.x.y-t.z.y.x\\
		&=&0.	
	\end{eqnarray*}
	This shows that the  Lie algebra $\G$ is two-step solvable.
\end{proof}
Recall that  an algebra $(\G,.)$  is called an \ita{LR-algebra}, if the product satisfies the identities
\[x.(y.z)=y.(x.z)\qquad and \qquad (x.y).z = (x.z).y\]
for all $x$, $y$, $z\in\G$. Indicating that in an LR-algebra the left and right multiplication operators commute i.e. $[\mathrm{L}_y,\mathrm{L}_x]=0$ and $[\mathrm{R}_y,\mathrm{R}_x]=0$. From (\ref{R8}) and (\ref{R9})  we obtain the following corollary
\begin{co}
	The SNLA $(\G,\omega)$, is  LR-algebra  if and only if  the Lie algebra $\G$ is two-step nilpotent.
\end{co}

\section{Examples and Low dimensions SNLA}

We give a complete classification of  four dimensional SNLA and  six dimensional   nilpotent  SNLA.
\subsection{Low dimensions SNLA}
Start with the non-abelian tow-dimensional Lie algebra $\mathrm{aff}(\R)=span\{e_1,e_2\}$, the  symplectic Lie algebra $(\mathrm{aff}(\R),e^{12})$ with  $[e_1,e_2]=e_2$ is SNLA and the Novikov product is $e_1.e_1=-e_1$,\; $e_2.e_1=-e_2$.

We use the classification of four-dimensional symplectic  Lie algebras given by Ovando in \cite{C} and the  classification of six-dimensional symplectic nilpotent Lie algebras given in \cite{K-G-M} (see also \cite{F} for a more recent list). One obtains after a direct computation.
\begin{pr}
	The four-dimensional SNLAs are listed below:
	\begin{center}
		{\renewcommand*{\arraystretch}{1.3}
			\begin{tabular}{|c|l|c|}
				\hline
				Case&Novikov structure&Symplectic structure\\
				\hline
				$\mathfrak{r} \mathfrak{h}_3$ & $e_1.e_2=e_3$, $e_2.e_2=-e_4$ &$e^{14}+e^{23}$\\
				\hline
				$\mathfrak{rr}_{3,0}$& $e_1.e_1=-e_1$, $e_2.e_1=-e_2$ &$e^{12}+e^{34}$\\
				\hline
				\multirow{2}{*}{$\mathfrak{d}_{4,1}$}&$e_1.e_2=e_3$, $e_1.e_4= -e_1$, $e_2.e_4=-e_2$&\multirow{2}{*}{$e^{12}-e^{34}$}\\
				&$e_3.e_4=-e_3$, $e_4.e_2=-e_2$, $e_4.e_4=-e_4$&
				\\
				\hline
				\multirow{2}{*}{$\mathfrak{r}^\prime_2$}&$e_1.e_1=-e_1$, $e_1.e_2=- e_2$, $e_2.e_1=-e_2$, $e_2.e_2=e_1$&\multirow{2}{*}{$e^{14}+e^{23}$}\\
				&$e_3.e_1=-e_3$, $e_3.e_2=-e_4$, $e_4.e_1=-e_4$, $e_4.e_2=e_3$&\\
				\hline
		\end{tabular}}
	\end{center}
	The six-dimensional nilpotent  SNLAs   are listed below:
	
	\begin{center}
		{\renewcommand*{\arraystretch}{1.3}
			\begin{tabular}{|c|l|c|}
				\hline
				Case&Novikov structure&Symplectic structure\\
				\hline
				\multirow{3}{*}{$L_{6,18}$}&$e_1.e_2=\frac{\lambda}{\lambda-1}e_4$, $e_1.e_3=\frac{\lambda-1}{\lambda}e_5$ &\\
				&$e_2.e_1=\frac{1}{\lambda-1}e_4$, $e_2.e_3=(1-\lambda)e_6$&$e^{16}+\lambda e^{25}+(\lambda-1)e^{34}$,
				\\
				&$e_3.e_1=-\frac{1}{\lambda}e_5$, $e_3.e_2=-\lambda e_6$&  $\lambda\in\mathbb{R}-\{0,1\}$\\
				\hline
				\multirow{4}{*}{$L_{6,18}$}&$e_1.e_1=-\frac{1}{2\lambda}e_4$, $e_1.e_2=\frac{1}{2}e_4$, $e_2.e_1=-\frac{1}{2}e_4$&\\
				&$e_2.e_2=-\frac{1}{2\lambda}e_4$, $e_1.e_3=\frac{2\lambda}{\lambda^2+1}(\lambda e_5-e_6)$ &$e^{15}+\lambda e^{16}-\lambda e^{25}+e^{26}-2\lambda e^{34}$, \\
				&$e_2.e_3=\frac{2\lambda}{\lambda^2+1}(e_5+\lambda e_6)$, $e_3.e_1=\frac{\lambda^2-1}{\lambda^2+1}e_5-\frac{2\lambda}{\lambda^2+1}e_6$& $\lambda\in\mathbb{R}-\{0,1\}$\\
				&$e_3.e_2=\frac{2\lambda}{\lambda^2+1}e_5+\frac{\lambda^2-1}{\lambda^2+1}e_6$&
				\\
				\hline
				\multirow{3}{*}{$L_{6,18}$}&$e_1.e_2=-e_4-2e_5$, $e_1.e_3=-e_5$, &\\
				&$e_2.e_1=-2e_4-2e_5$, $e_2.e_2=\frac{1}{2}e_6$& $e^{35}-e^{16}+e^{25}+2e^{34}$\\
				&$e_2.e_3=\frac{1}{2}e_6$, $e_3.e_1=-2e_5$, $e_3.e_2=-\frac{1}{2}e_6$&\\
				\hline	\multirow{2}{*}{$L_{6,21}$}&$e_1.e_1=e_3$, $e_1.e_3=-e_5$, $e_2.e_1=-e_4$&
				\multirow{2}{*}{$\pm(e^{16}+e^{25}-e^{34})$}\\
				&$e_2.e_3=e_6$, $e_3.e_1=-e_5$, $e_4.e_1=-e_6$&\\
				
				\hline
				\multirow{1}{*}{$L_{6,23}$}&$e_1.e_1=e_4$, $e_1.e_2=e_5$, $e_2.e_2=-e_6$, $e_3.e_1=-e_6$&$e^{16}+e^{25}+e^{34}$\\\hline
				\multirow{1}{*}{$L_{6,23}$}&$e_1.e_2=e_5$, $e_1.e_3=e_6$, $e_2.e_3=\pm e_4$, $e_3.e_2=\pm e_4$&$e^{14}\pm(e^{26}+ e^{35})$\\
				\hline
				$L_{6,25}$&$e_1.e_1=e_5$, $e_2.e_1=-e_6$&$e^{16}+e^{24}-e^{35}$\\
				\hline
		\end{tabular}}
	\end{center}
\end{pr}

\begin{remark}
	\begin{enumerate}
		\item The symplectic Lie algebras $\mathfrak{d}_{4,1}$ and $\mathfrak{r}^\prime_2$ are Frobenius Lie algebras (their symplectic forms are exacts).
		\item The Lie  algebras $L_{6,18}$, $L_{6,23}$ and $L_{6,25}$  are the only  two-step nilpotent SNLA. It is well known that \cite{B-D} all two-step nilpotent Lie algebra admits a Novikov structure, we notice that this result is no longer correct in the case of SNLA.
		\item Being NSLA is a property that depends on both the Lie structure and the symplectic structure. For example, the left-symmetric product associated with the symplectic Lie algebra $(\mathfrak{d}_{4,1},e^{12}-e^{34}+e^{24})$ is not Novikov.
	\end{enumerate}
\end{remark}
We obtain the following corollary 
\begin{co}
	Any nilpotent SNLA of dimension less than or equal to six is at most a $3$-step nilpotent Lie algebra.
\end{co}

\subsection{Symplectic Novikov cotangent Lie algebra}
Let $(\h,.)$ be a left-symmetric algebra, 
for any $x\in\h$, $\ell_x,\,r_x :\h\lr\h$ denote the left and the right multiplication by $x$ given by $\ell_xy=x.y$ and $r_xy=y.x$, respectively  and $\mathrm{ad}_x$ is the endomorphism of $\h$ given by $\mathrm{ad}_xy= [x,y]$. 

A \ita{symplectic cotangent Lie algebra}  $(\h^\ast\times_{\ell}\h,\omega)$ of $\h$, is the vector space $\h^\ast\times\h$ endowed with the Lie bracket 
\begin{equation*}
[(\alpha,x),(\beta,y)]=\big(\ell^t_y \alpha-\ell^t_x \beta,[x,y]\big),\qquad x,y\in\h,\; and\;\alpha,\beta\in\h^*
\end{equation*}
and non-degenerate $2$-cocycle
\begin{equation*}
\omega\big((\alpha,x),(\beta,y)\big)=\beta(x)-\alpha(y),\qquad x,y\in\h,\; and\;\alpha,\beta\in\h^*.
\end{equation*}
\begin{pr}
	The symplectic cotangent Lie algebra  $(\h^\ast\times_{\ell}\h,\omega)$ is an SNLA if and only if $(\h,.)$ is Novikov associative algebra. 
	
\end{pr}
\begin{proof}
	We denote by  $\odot$ the left symmetric product associated with  $(\h^\ast\times_{\ell}\h,\omega)$. 	
	For any $x,y,z\in\h$ and for any $\alpha,\beta,\gamma\in \h$, we have
	\begin{eqnarray*}
		\omega\big((\alpha,x)\odot (\beta,y),(\gamma,z)\big)&=&-\omega\big((\beta,y),[(\alpha,x),(\gamma,z)]\big)\\
		&=&-\omega\big((\beta,y),(\ell^t_z\alpha-\ell^t_x\gamma,[x,z])\big)\\
		&=&ad^t_x\beta(z)+\ell^t_x\gamma(y)-r^t_y\alpha(z).
	\end{eqnarray*}
	
	By asking $(\alpha,x)\odot (\beta,y)=(\lambda,s)$ and for $\gamma=0$, we get that $\lambda=-ad^t_x\beta+r^t_y\circ\alpha$, similarly, if $z=0$, we obtain that $s=\ell_x y=x.y$. Therefore
	\begin{equation*}
	(\alpha,x)\odot (\beta,y)=\big(-ad^t_x\beta+r^t_y\alpha,x. y\big).
	\end{equation*}
	On the other hand, the product $\odot$ is Novikov then
	\begin{eqnarray*}
		\big((\gamma,z)\odot (\beta,y)\big)\odot (\alpha,x)=\big((\gamma,z)\odot (\alpha,x)\big)\odot (\beta,y)
	\end{eqnarray*}
	a direct computation yields
	\begin{equation*}
	\Big(-ad^t_{z.y}\alpha-r^t_x\circ ad^t_z\beta+r_x^t\circ r_y^t(\gamma),(z.y).x\Big)=\Big(-ad^t_{z.x}\beta-r^t_y\circ ad^t_z\alpha+r_y^t\circ r_x^t(\gamma),(z.x).y\Big)
	\end{equation*}		
	then 
	\[[r_x,r_y]=0 \qquad and\qquad \mathrm{ad}_{x.y}z=\mathrm{ad}_y\circ r_x(z),\]  for all $x,y,z\in\h$. This is equivalent to,
	\[[r_x,r_y]=0 \qquad and\qquad ass(x,z,y)=0,\]
	for all $x,y,z\in\h$. Hence the result.
\end{proof}

\begin{exem}
	\begin{enumerate}
		\item All $n$-dimensional Novikov algebra with an abelian Lie algebra   define an $2n$-dimensional symplectic Novikov cotangent Lie algebra.

		\item The following table gives all three-dimensional Novikov associative algebras
		
		{\renewcommand*{\arraystretch}{1.3}
			\begin{tabular}{cl}
				$A_{3,2}\quad :$& $e_1^2=e_2$  \\
				$A_{3,3}\quad :$& $e_1^2=e_2$, $e_1.e_2=e_2.e_1=e_3$ \\
				$A_{3,4}\quad :$& $e_1^2=e_3$, $e_2^2=e_3$ \\
				$A_{3,5}\quad :$& $e_1^2=-e_3$, $e_2^2=e_3$ \\
				$g_{3,1}\quad :$& $e_1^2=e_2$, $e_1.e_2=(a+1)e_3$, $e_2.e_1=ae_3$ \\
				$g_{3,2}\quad :$& $e_1^2=ae_3$, $e_1.e_2=e_3$, $e_2^2=e_3$ \\
				$g_{3,3}\quad :$& $e_1.e_2=\frac12e_3$, $e_2.e_1=-\frac12e_3$.\\
		\end{tabular}}
	\end{enumerate}	
	For the list of three-dimensional real Novikov algebras  see \cite{B-G}. 
\end{exem}

Recall  the following well-known construction of Novikov algebra (see \cite{B-G}). This construction can generate many examples of SNLAs.
\begin{pr}
	Let $(\G,.)$ be an associative, commutative algebra and $D$ a derivation of $(\G,.)$,
	i.e., satisfying $D(x.y) = D(x).y + x.D(y)$. Then the product $x\ast y=x.D(y)$ is Novikov.
\end{pr}
In particular, it defines a Novikov structure on the Lie algebra  given by
\begin{equation*}
[x, y] := x \ast y-y\ast x = x.D(y)-y.D(x)\qquad\forall x,y\in\G.
\end{equation*}
Note that $D\in Der(\G,.)$ implies $D\in Der(\G,[\,,\,])$.
Thus, we obtain the following result.

\begin{co} \label{c2}
	Let $(\h,.) $ be an associative, commutative algebra  and $ D\in Der(\h,.)$. Then the symplectic cotangent Lie algebra $ (\h^\ast\times_{\ell}\h,\omega) $ is an SNLA if and only if 
	\begin{equation} \label{7}	
	x.y.D ^ 2 (z) = 0, \qquad \forall x, y, z \in \h. 
	\end{equation}
	In particular, if $ D^2=0 $, then $(\h^ \ast \times_{\ell}\h, \omega) $ is an SNLA.
\end{co}

\section{Reduction and SNLA  Oxidation}

Let  $(\G,\omega)$ be a symplectic Lie algebra and $\h$  an isotropic (i.e.  $\omega_{D(\G)\times D(\G)}=0$ or in other words $\h\subseteq \h^\perp$) ideal. The
orthogonal $\h^\perp$ is a subalgebra of $\G$ which contains $\h$, and therefore $\omega$ descends
to a symplectic form $\overline{\omega}$ on the quotient Lie algebra $\overline{\G}=\h^\perp/\h$. The symplectic Lie algebra $(\overline{\G},\overline{\omega})$ is called the \ita{symplectic reduction} of $(\G,\omega)$ with respect to the isotropic ideal $\h$. Recall that, a symplectic Lie algebra  is called \ita{completely reducible} if it can be symplectically reduced (in several steps) to the trivial symplectic algebra. Note that, every four-dimensional symplectic Lie algebras,    nilpotent symplectic Lie algebras  and  completely solvable symplectic Lie algebras are  completely reducible. Moreover, an \ita{irreducible symplectic Lie algebra} is a symplectic Lie algebra which does not admit a reduction, that is,  it does not have a non-trivial isotropic ideal (for more details see \cite{B-C}).

\begin{Le}\label{L1}
	Let $(\G,\omega)$ be an SNLA, then all   symplectic reductions $(\overline{\G},\overline{\omega})$ are SNLA. 
\end{Le}
\begin{proof}
	It suffices to remark  that the left-symmetric  product associated with $\overline{\omega}$ is given by $\overline{x}.\overline{y}=\overline{x.y}$ for $x,y\in\h^\perp$.
\end{proof}

\begin{theo}\label{t2} An  irreducible no  trivial  symplectic Lie algebra is never an SNLA.
\end{theo}
\begin{proof} Let $(\G,\omega)$ be a irreducible symplectic Lie algebra, Theorem $3.16$ \cite{B-C}  gives a characterization of irreducible symplectic Lie algebra, more precisely it shows that the commutator ideal $D(\G)$ is a maximal abelian ideal of $\G$, which is non-degenerate with respect to $\omega$, the symplectic Lie algebra $(\G,\omega)$ is an orthogonal semi-direct sum of	an abelian symplectic subalgebra $(\h,\omega_{|\h})$ and the ideal $(D(\G),\omega_{|D(\G)})$  and we can choose a basis  $\{f_1,\overline{f_1},...,f_h,\overline{f_h},e_{1,1},e_{2,1},...,e_{1,m},e_{2,m}\}$ of $\G$, such that $\{f_1,\overline{f_1},...,f_h,\overline{f_h}\}$ is  a basis  of $\h$ and $\{e_{1,1},e_{2,1},...,e_{1,m},e_{2,m}\}$
	is  a basis  of $D(\G)$. In this basis  the Lie bracket  is given for $\lambda^k_i$, $\overline{\lambda}^k_i\in\R$ by
	\begin{equation*}
	[f_i,e_{1,k}]=-\lambda^k_ie_{2,k} \qquad[f_i,e_{2,k}]=\lambda^k _ie_{1,k}
	\end{equation*}
	\begin{equation*}
	[\overline{f_i},e_{1,k}]=-\overline{\lambda}^k_ie_{2,k} \qquad[\overline{f_i},e_{2,k}]=\overline{\lambda}^k _ie_{1,k}
	\end{equation*}
	and the symplectic form is given by
	\begin{equation*}
	\omega=\sum\limits_{i=1}^h f^i\wedge \overline{f}^i+\sum\limits_{i=1}^me^{1,k}\wedge{e}^{2,k}.
	\end{equation*}
	From (\ref{5}) a direct computation yields the left-symmetric  product    associated with symplectic Lie algebra $(\G,\omega)$.
	\begin{equation*}
	\begin{cases}
	e_{1,k}.e_{1,k}=e_{2,k}.e_{2,k}=\sum\limits_{i=1}^m\lambda^k_i\overline{f_i}-\overline{\lambda}^k_if_i\\
	f_i.e_{1,k}=-\lambda^k_ie_{2,k},\quad\overline{f_i}.e_{1,k}=-\overline{\lambda}^k_ie_{2,k},\quad 1\leq i\leq h\,\;and\;\;1\leq k\leq m\\
	f_i.e_{2,k}=\lambda^k_ie_{1,k},\quad\overline{f_i}.e_{2,k}=\overline{\lambda}^k_ie_{1,k}.\\
	\end{cases}
	\end{equation*}
	For $1\leq j\leq h$ and $1\leq k\leq m$, we have
	\begin{eqnarray*}
		(f_j.e_{2,k}).e_{1,k}-(f_j.e_{1,k}).e_{2,k}&=&\lambda_j^ke_{1,k}.e_{1,k}+\lambda_j^ke_{2,k}.e_{2,k}\\
		&=&2\lambda_j^k\Big(\sum\limits_{i=1}^m\lambda^k_i\overline{f_i}-\overline{\lambda}^k_if_i\Big)\\
		&=&2\Big((\lambda_j^k)^2\overline{f_j}-\lambda_j^k\overline{\lambda}_j^kf_j\Big)+
		2\lambda_j^k\Big(\sum\limits_{\substack{i=1\\i\not=j}}^m\lambda^k_i\overline{f_i}-\overline{\lambda}^k_if_i\Big).
	\end{eqnarray*}
	Thus, if $(\G,\omega) $ is an SNLA we obtain  $\lambda_j^k=0$ for $1\leq j\leq h$ and $1\leq k\leq m$, a seminar calculation of $(\overline{f}_j.e_{2,k}).e_{1,k}-(\overline{f}_j.e_{1,k}).e_{2,k}$ implies that $\overline{\lambda}_j^k=0$ for $1\leq j\leq h$ and $1\leq k\leq m$.  Then an SNLA  is irreducible if and only if it is trivial.
\end{proof}

According to lemma \ref{L1} we obtain the following corollary 
\begin{co}
	Every SNLA is	completely reducible in particular an SNLA admits an isotropic ideal.
\end{co}
The following Lemma is an immediate consequence of $(\ref{R8})$.

\begin{Le}
	Let $(\G,\omega)$ be an SNLA. Then the commutator ideal $D(\G)=[\G,\G]$ is isotrope.
\end{Le}

\begin{pr}
	Let $(\G,\omega)$ be a $p$-step nilpotent SNLA. Then  we have the estimate
	\[p\leq\dim D(\G)+1\leq \frac{1}{2}\dim\G+1.\]	
\end{pr}
\begin{proof}
	Let $(\G,\omega)$ be a $p$-step nilpotent SNLA and 	let
	\begin{eqnarray*}\G=C^0(\G) \supsetneq C^1(\G)=D(\G)\supsetneq...\supsetneq C^p(\G)=0
	\end{eqnarray*}
	with $C^k(\G) = [\G, C^{k-1}(\G)]$ be a the descending central sequence of $\G$, which shows that
	\[\underbrace{\dim D(\G)>...>0}_{p-times}.\]
	Then $p\leq\dim D(\G)+1$ and we deduce from Lemma $2$ that  $\dim D(\G)\leq  \frac{1}{2}\dim\G$.
\end{proof}
This Proposition with Proposition $1$  shows in particular the following corollary.
\begin{co}
	Any filiform  or quasi-filiform Lie algebra is never an SNLA.
\end{co}

There exists a process to construct the symplectic Lie algebras. It is often called the construction by double extension and it was described by  Medina and  Revoy \cite{M-R}. This construction also called \ita{central symplectic oxidation}   was studied and developed by Baues and Cortés in \cite{B-C}.

Let $(\G,\omega)$ be a symplectic Lie algebra and  let $\varphi$ be a derivation of this Lie algebra. Then the bilinear map $\omega_\varphi$ on $\G$ given by
\begin{equation*}
\omega_\varphi(x,y)= \omega(\varphi x,y)+\omega(x,\varphi y)
\end{equation*}

is a closed 2-form on $\G$. Then we can define a one-dimensional central extension $\G_1 =\G\oplus\langle h\rangle$ (vector space direct sum of $\G$ with one-dimensional vector spaces $\langle h\rangle$) of $\G$ by
\begin{equation*}
[x+t_1h,y+t_2h]_{\G_1}= [x,y]+\omega_\varphi(x,y)h.\quad x,y\in\G,\;t_1, t_2\in\R.
\end{equation*}
It is easy to verify that the bilinear form $\omega_{\varphi,\varphi}$ defined by
\begin{equation*}
\omega_{\varphi,\varphi}(x,y)= \omega_\varphi(\varphi x,y)+\omega_\varphi(x,\varphi y)
\end{equation*}
is also a   closed 2-form on $\G$. Assume that our closed form is exact, that is there exists $\lambda\in\G^*$ with $\omega_{\varphi,\varphi}(x,y)=-\lambda[x,y]$. We can now define a new derivation $\varphi_1$ of the Lie algebra $\G_1$ by
\begin{equation*}
\varphi_1(h)=0\quad and\quad\varphi_1(x)=-\varphi(x)-\lambda(x)h,\; \forall x\in\G.
\end{equation*}
We consider the one-dimensional extension by derivation $\G_2 =\langle \xi \rangle \oplus\G_1$ of $\G_1$. We
define a non-degenerate two-form $\Omega$ on $\G_2$ by
\[\begin{cases}
\Omega(x,y)=\omega(x,y)\qquad \forall x,y\in\G\\
\Omega(\xi,x)=\Omega(h,x)=0\quad \forall x\in\G\\
\Omega(\xi,h)=1.
\end{cases}\]
\begin{Def}
	Let $(\G,\omega)$ be a symplectic Lie algebra, $\varphi\in Der(\G)$ a derivation and $\lambda\in\G^*$ a one form such that $\omega_{\varphi,\varphi}(x,y)=-\lambda([x,y])$  for $x,y\in\G$. The symplectic Lie algebra $(\G_2,\Omega)$ is called \ita{central symplectic oxidation} with respect $(\varphi,\lambda)$.
\end{Def} 
\begin{remark}
	\begin{enumerate}
		\item It is clear that the Lie bracket in $\G_2 =\langle \xi \rangle \oplus\G\oplus\langle h \rangle $ is given by
		\[\begin{cases}
		[x,y]_{\G_2}=[x,y]+\omega_\varphi(x,y)h,\quad \forall x,y\in G\\
		[\xi,x]_{\G_2}=\varphi(x)+\lambda(x)h,\quad \forall x\in G.
		\end{cases}\]
		In particular, $\langle h \rangle$
		is a central ideal of $\G_2$.
		\item Let $(\G,\omega)$ be a symplectic Lie algebra and $\varphi\in Der(\G)$ a
		derivation. Then there exists a symplectic oxidation if and only	if the cohomology class
		$[\omega_{\varphi,\varphi}]\in H^2(\G)$	vanishes.
		\item Any nilpotent symplectic Lie algebra is a symplectic oxidation of a symplectic nilpotent Lie algebra.
	\end{enumerate}
\end{remark}
\begin{Le}
	Let $(\G,\omega)$ be a symplectic Lie algebra, $\varphi\in Der(\G) $ and $ \zeta\in\G $ such that $\mathrm{ad}_\zeta=\varphi^2 $ and $ \varphi\circ\varphi^*= 0$. Then the obstruction $ [\omega_ {\varphi,\varphi}]\in H^2 (\G)$ to the existence of a symplectic oxidation with respect to $(\varphi, \omega (\zeta,.))$  vanishes.
\end{Le}
\begin{proof} Let $x$ and $y\in\G$. Then
	\begin{eqnarray*}
		\omega_{\varphi,\varphi}(x,y)&=& \omega(\varphi^2 x,y)+\omega(\varphi x,\varphi y)+\omega(\varphi x,\varphi y)+\omega(x,\varphi^2 y)\\
		&=&\omega(\mathrm{ad}_\zeta x,y)+\omega(x,\mathrm{ad}_\zeta y)\\
		&=&-\omega(\zeta,[x,y])\\
		&=&\lambda([x,y]),
	\end{eqnarray*}
	with $\lambda\in\G^*$ given by $\lambda(x)=\omega(-\zeta,x)$.
\end{proof}

Let's keep the notation above, the following theorem gives conditions on  $(\varphi,\lambda)$ so that the symplectic oxidation is an SNLA.
\begin{theo}\label{t3}
	Let $(\G,\omega)$ be a symplectic Lie algebra. The  central symplectic oxidation $\langle \xi \rangle \oplus\G\oplus\langle h \rangle$ with respect to $(\varphi,\lambda)$ is an  SNLA if $(\G,\omega)$ does and the  following properties are verified
	\begin{enumerate}
		\item  $\varphi\circ \mathrm{L}_x=\mathrm{ad}_x\circ \varphi^* $ and $\mathrm{R}_x\circ \varphi^*=\mathrm{L}_{\varphi^*(x)}$\qquad $\forall x\in\G$.
		\item $\mathrm{R}_x\circ(\varphi+\varphi^*)=(\varphi+\varphi^*)\circ \mathrm{R}_x$ \qquad \quad  $\forall x\in\G$.
		\item $\varphi\circ \varphi^*=0$,  $\mathrm{ad}_\zeta=\varphi^2$ and $\zeta\in \mathcal \ker(\varphi)$,
	\end{enumerate}
	with  $\zeta\in\G$ such that  $\lambda =\omega(\zeta,.)$.
\end{theo}
\begin{proof}
	
	Let $"\ast"$ be the left-symmetric product associated with $(\G, \omega)$ and $"."$ the one associated with the symplectic oxidation. We have
	
	\begin{eqnarray}\label{s11}\begin{cases}
	x.y=x*y+\omega(\varphi x,y)h\\
	\xi.x=-\varphi^*x\\
	x.\xi=-(\varphi+\varphi^*)x-\lambda(x)h\\
	\xi.\xi=-\zeta.
	\end{cases}	\end{eqnarray}
	Using this we obtain 
	
	\begin{eqnarray}\label{s}
	(x.y).z&=&(x\ast y)\ast z+\omega(\varphi (x\ast y),z)h\\
	(\xi.x).y&=&-\varphi^*(x)*y-\omega(\varphi\varphi^*(x),y)h\\
	(x.\xi).y&=&-(\varphi+\varphi^*)(x)*y-\omega(\varphi(\varphi+\varphi^*)(x),y)h\\
	(x.y).\xi&=&-(\varphi+\varphi^*)(x*y)-\lambda(x*y)h\\
	(\xi.\xi).x&=&-\zeta*x-\omega(\varphi(\zeta),x)h\\
	(\xi.x).\xi&=&(\varphi^*)^2(x)+\lambda((\varphi^*(x))h.
	\end{eqnarray}
	By using $(11)$, we can see  that  $(x.y).z=(x.z).y$  is equivalent to
	\begin{eqnarray*}
		(x\ast y)\ast z=(x\ast z)\ast y\quad and\quad   \omega(\varphi (x\ast
		y),z)=\omega(\varphi (x\ast z),y).
	\end{eqnarray*}
	for any $x$, $y$, $z\in\G$, which is equivalent to $(\G,\omega)$ is an SNLA and $\mathrm{L}_x=\mathrm{ad}_x\circ \varphi^*$.
	
	By using $(12)$, we can see  that  $(\xi.x).y=(\xi.y).x$  is equivalent to
	\begin{eqnarray*}
		\varphi^*(x)*y=\varphi^*(y)*x\quad and\quad
		\omega(\varphi\varphi^*(x),y)=\omega(\varphi\varphi^*(y),x)
	\end{eqnarray*}
	for any $x$, $y\in\G$, which is equivalent to $\mathrm{R}_y\circ \varphi^*=\mathrm{L}_{\varphi^*(y)}$ and   $\varphi\circ \varphi^*=0$.
	
	Now, by using $(13)$ and $(14)$, we can see easily that $(x.\xi).y=(x.y).\xi$  is equivalent to
	\begin{eqnarray*}
		(\varphi+\varphi^*)(x*y)=(\varphi+\varphi^*)(x)*y \quad and\quad  \omega(\varphi^2(x),y)=\lambda(x*y)=\omega(\mathrm{ad}_\zeta x,y)
	\end{eqnarray*}
	for any $x$ and $y\in\G$,  which is equivalent to $(\varphi+\varphi^*) \circ \mathrm{R}_y=\mathrm{R}_y\circ (\varphi+\varphi^*)$ and $\varphi^2=\mathrm{ad}_\zeta$.
	
	Finally, by using $(15)$ and $(16)$, we can see  that  $(\xi.\xi).x=(\xi.x).\xi$  is equivalent to
	\begin{eqnarray*}
		(\varphi^*)^2=-\mathrm{L}_\zeta=\mathrm{ad}_\zeta^*\qquad and \qquad \varphi(\zeta)=0.
	\end{eqnarray*}
	Thus the proposition follows.
\end{proof}

\begin{pr}
	Let $(\G,\omega)$ be a $2n$-dimensional SNLA with non trivial center. Then, $(\G,\omega)$ is a central SNLA oxidation of a $(2n-2)$-dimensional SNLA.
\end{pr}

\begin{proof}
	Let $(\G,\omega)$ be a $2n$-dimensional SNLA with a non trivial center, $\langle h\rangle$ a one dimensional central ideal and $(\overline{\G},\overline{\omega})$ the symplectic reduction
	with respect to $\langle h\rangle$. We may choose $\xi\in\G$ such that $\omega(\xi,h)=1$ and $\G =(\langle \xi \rangle \oplus\overline{\G}\oplus\langle h \rangle,\overline{\omega}) $ is the  central symplectic oxidation with respect $(\varphi_\xi,\lambda_\omega)$, with  $\varphi_{\xi}=\mathrm{ad}_{\xi}|_{\overline{\G}}$ is the restriction of adjoint operator and $\lambda_\omega(x)=\omega(\xi,[\xi,x])$ for $x\in\G$   (see \cite{B-C} Proposition 2.16.). Remark that $\varphi_{\xi}=-\mathrm{L}_\xi|_{\overline{\G}}$,   then  (\ref{s11}) becomes for all $x$, $y\in\overline{\G}$
	\begin{eqnarray*}\label{s2}\begin{cases}
			x.y=x*y+\omega(\mathrm{ad}_{\xi}|_{\overline{\G}} x,y)h\\
			\xi.x=\mathrm{L}_\xi|_{\overline{\G}}x \\
			x.\xi=\mathrm{R}_\xi|_{\overline{\G}}x-\lambda_\omega(x)h\\
			\xi.\xi=-\zeta\\
	\end{cases}	\end{eqnarray*}
	
	The proof is completed by verifying that $(\varphi_\xi,\lambda_\omega)$ verified the conditions of Theorem $\ref{t3}$ .
\end{proof}
\section{Geometry of symplectic Novikov Lie group}
Let $(G,\omega^+)$ be a symplectic Lie group, (i.e., a Lie group $G$ endowed with a
left invariant symplectic form $\omega^+$). We denote by $\G$ the Lie algebra of $G$, $\omega=\omega^+(e)$, with $e$ the unit of $G$. Let $\nabla$ be a
left invariant linear connection given by the formula
\begin{eqnarray}\label{17}
\nabla_{x^+}y^+=(x.y)^+
\end{eqnarray}	

with $x^+$ denotes the left invariant vector field on $G$ whose value at $e$ is $x\in\G$ and    $"."$ is a  left symmetric product associated with $(\G,\omega)$.
It is well known that $\nabla$ defines a left invariant affine structure on $G$, (i.e a left invariant flat and torsion free connection). This affine structure is called associated to the symplectic structure.

We will say subsequently that a symplectic Lie group $(G,\omega^+)$ is \ita{symplectic Novikov Lie group} if the symplectic Lie algebra $(\G,\omega )$ is an SNLA.

\begin{pr}	
	Let $(G,\omega^+)$ be a symplectic Novikov Lie group. The affine structure associated to $\omega^+$, noted by $\nabla$, satisfies the following properties
	\begin{enumerate}
		\item  $\nabla$ is bi-invariant.
		\item  $\nabla$ is  completed if and only if the Lie algebra $\G$ is nilpotent.
	\end{enumerate}	
	\begin{proof}
		\begin{enumerate}
			\item The first point follows from a general result (see, for example, \cite{M} Proposition 1-1) which states that in a Lie group an affine connection is bi-invariant if and only if  its associated left-symmetric product is associative.
			
			\item  It is known, that a left-invariant affine structure on a Lie group $G$ is complete if and only if the 	right multiplications $R_x$ in the corresponding left symmetric algebra $(\G,.)$  are
			nilpotent for all $x\in\G$. Iterating $(\ref{5})$ we have
			\[\omega(\mathrm{L}_x^ky,z)=(-1)^k\omega(y,\mathrm{ad}_x^kz)\]
			for all positive integers $k$, where $x$ , $y$ , $z \in\G$. In particular    the Lie algebra $\G$ is nilpotent if and only if the operators $\mathrm{L}_x$ are nilpotent for all $x\in\G$. On the other hand, a direct calculation gives
			\begin{eqnarray*}
				\mathrm{L}_x\circ \mathrm{ad}_x=\mathrm{ad}_x\circ \mathrm{L}_x=0,
			\end{eqnarray*}
			for all $x\in\G$. Furthermore we have
			\begin{eqnarray*}
				\mathrm{R}_x^k=\mathrm{ad}_x^k+(-1)^k\mathrm{L}_x^k,
			\end{eqnarray*}	
			then $(2)$ holds.	
		\end{enumerate}	
	\end{proof}	
\end{pr}

Let $(M,\Omega)$ be a  smooth $2n$-dimensional symplectic manifold (i.e. $\Omega$ is a closed nondegenerate 2-form on $M$). A symplectic connection on $(M,\Omega)$ is a torsion free
linear connection $\nabla$  on $M$ for which the symplectic 2-form $\omega$ is parallel i.e. $\nabla\Omega=0$. Recall that $\cite{B-C-G}$ from any linear connection $\nabla$, we can build a symplectic connection. More precisely, let the tensor field $\mathrm{N} $ given by the relation
\begin{equation*}
\nabla_X\Omega(Y,Z)=\Omega(\mathrm{N}(X,Y),Z)
\end{equation*}
for all vector fields $X$, $Y$ and $Z$ on $M$. Then the linear connection given by

\begin{equation*}
\nabla^\Omega_XY=\nabla_XY+\frac13 \mathrm{N}(X,Y)+ \frac13\mathrm{N}(Y,X)
\end{equation*}
is  torsion free and preserves $\Omega$.

Now, let $(G,\omega^+)$ be a symplectic Lie group. By applying the above construction to a symplectic Lie group, we get  the following proposition 
\begin{pr}
	Let $(G,\omega^+)$ be a symplectic Lie group and $\nabla$ be the affine structure associated to $\omega^+$. The symplectic  connection $\nabla^{\omega}$ associated with $\nabla$ is given (in the neutral element) by
	\begin{equation}\label{p19}
	\nabla^{\omega} _ {x} {y} = \frac{1}{3}(\mathrm{ad}_x y-\mathrm{ad}^\ast_x y \ ) 
	\end{equation}
	for all $ x, y \in\G $.
\end{pr}

The connection $\nabla^\omega$ depends only on the Lie group and $\omega^+$ was first introduced by Benayadi and Boucetta in 
\cite{B-B} then appears in another articles (see for example  \cite{Fo} and \cite{A-S}).

\begin{remark}
	All Lie algebras two-step nilpotent admits $ x.y = \frac12[x, y] $ as a Novikov structure. This structure is not an SNLA. In fact, Note that the affine structure $\nabla$ associated to $\omega^+$  is symplectic unless $G$ is abelian. On the other hand a simple calculation gives $\nabla^\omega=\nabla=\frac12[x^+, y^+]$.
\end{remark}

Recall that, the curvature tensor is then described in terms of the map
\begin{equation*}
\begin{array}{rcl}
\mathcal{K} :\quad  \G\times\G& \longrightarrow&\mathfrak{gl}(\G) \\
(x,y) & \longmapsto&  \mathcal{K}(x,y)=[\nabla_x, \nabla_y]-\nabla_{[x,y]},
\end{array}
\end{equation*}
and $(\G,\nabla)$  is called flat if $\mathcal{K}= 0$. 	The Ricci tensor is the symmetric tensor $ric$ given by
\[ric(x,y) = tr(z\longmapsto \mathcal{K}(x,z)y).\]
\begin{theo}
	Let $(G,\omega^+)$ be a symplectic Novikov Lie group. For any $x$, $y\in\G$, we have	
	\begin{enumerate}	
		\item The curvature tensor of the connection $\nabla^ {\omega} $ is given (in the neutral element) by
		\begin{equation*}
		\mathcal{K}(x,y)=-\frac{2}{9}\mathrm{ad}_{[x,y]}.
		\end{equation*}
		In particular, $(G,\omega^+)$ is flat if and only if $\G$ is  two-step nilpotent.
		\item The Ricci tensor is given (in the neutral element) by
		\begin{equation*}
		ric(x,y)=\frac{2}{9}tr\big(\mathrm{ad}_x\circ \mathrm{ad}_y\big).
		\end{equation*}
	\end{enumerate}
\end{theo}

\begin{proof}
	\begin{enumerate}	
		\item	For any $x$, $y$ and $z\in\G$, we have
		\begin{eqnarray*}
			\mathcal{K}(x,y)z
			&=& ([\nabla^\omega_x, \nabla^\omega_y]-\nabla^\omega_{[x,y]})z\\
			&=& \nabla^{\omega}_{x}(\frac{2}{3}y.z-\frac{1}{3}z.y)-\nabla^{\omega}_{y}(\frac{2}{3}x.z-\frac{1}{3}z.x)-\nabla^{\omega}_{x.y}z+\nabla^{\omega}_{y.x}z\\
			&=&\frac{4}{9}(x.y.z-y.x.z)+\frac{2}{9}(x.z.y-x.z.y)-\frac{2}{9}(y.z.x-y.z.x)\\
			&&\frac{1}{9}(z.y.x-z.x.y)+\frac{2}{3}(y.x.z-x.y.z)+\frac{1}{3}(z.x.y-z.y.x)\\
			&=&-\frac{2}{9}\mathrm{L}_{[x,y]}z+\frac{1}{9}R_{[x,y]}z.
		\end{eqnarray*}
		The result follows immediately from (\ref{R8}) and (\ref{R9}).
		\item We fix a symplectic basis $(e_i,f_i)$ of $\G$ (i.e $\omega=\sum\limits_i e^i\wedge f^i$) with dual basis $(e^i,f^i)$, by use of the Einstein summation convention we have
		\begin{eqnarray*}
			ric(x,y)&=&\mathrm{tr}\big(z\longmapsto\mathcal{K}(x,z)y\big)\\
			&=&\omega\big(\mathcal{K}(x,e_i)y,e^i\big)-\omega\big(\mathcal{K}(x,f_i)y,f^i\big)\\
			&=&-\dfrac{2}{9}\big(\omega([x,e_i]y,e^i)-\omega([x,f_i]y,f^i)\big)\\
			&=&-\dfrac{2}{9}\big(\omega(x.y.e_i-e_i. x.y,e^i)-\omega(x.y.f_i-f_i. x.y,f^i)\big)\\
			&=&-\dfrac{2}{9}\big(\mathrm{tr}(\mathrm{L}_{x.y})-\mathrm{tr}(R_{x.y})\big).
		\end{eqnarray*}
		A direct computation yields $\mathrm{tr}(\mathrm{R}_{x.y})=2\mathrm{tr}(\mathrm{L}_{x.y})$ and the associativity gives $\mathrm{L}_{x.y}=\mathrm{L}_x\circ \mathrm{L}_y$, for all $x$, $y\in\G$ then 
		\[ric(x,y)=\frac{2}{9}tr\big(\mathrm{L}_x\circ \mathrm{L}_y\big)=\frac{2}{9}\mathrm{tr}\big(\mathrm{ad}_x\circ \mathrm{ad}_y\big).\]
	\end{enumerate}
\end{proof}
\begin{remark}
	For a general study of the Ricci curvature of the symplectic conection associated with the affine connection of a symplectic Lie group see \cite{A-S}.	
\end{remark}
\section*{Acknowledgments:} 
The authors are sincerely grateful to Professor Boucetta for his involvement in improving this article.

\end{document}